\documentclass[12pt,a4paper]{article}

\usepackage{amsfonts,amsmath,amssymb,amsthm,amscd,epsfig,geometry,multicol,multirow,pifont,mathrsfs}
\usepackage{tikz}
\usetikzlibrary{positioning,automata,calc}
\usetikzlibrary{matrix,arrows,decorations.pathmorphing}
\tikzset{node distance=2cm, auto}

\theoremstyle{plain}
\newtheorem{thm}{Theorem}[section]
\newtheorem{defn}[thm]{Definition}
\newtheorem{lemma}[thm]{Lemma}

\newtheorem{prop}[thm]{Proposition}
\newtheorem{cor}[thm]{Corollary}
\newtheorem{remark}[thm]{Remark}

\DeclareMathOperator{\Fil}{Fil} 
\DeclareMathOperator{\GL}{GL} 
\DeclareMathOperator{\Aut}{Aut} 
\DeclareMathOperator{\End}{End} 
\DeclareMathOperator{\Gal}{Gal} 
\DeclareMathOperator{\Lie}{Lie} 
\DeclareMathOperator{\Hom}{Hom} 	
\DeclareMathOperator{\Ker}{Ker} 
\DeclareMathOperator{\Q}{\mathbb{Q}} 
\DeclareMathOperator{\Z}{\mathbb{Z}} 

\title{\textbf{On the vanishing of cohomologies of $p$-adic 
Galois representations associated with elliptic curves}}
\author{\textbf{Jerome T. Dimabayao} 
\date{}
}

\begin{document} 
\setlength\parindent{24pt}
\maketitle

\begin{abstract} 
Let $K$ be a $p$-adic field and $E$ an elliptic curve 
over $K$ with potential good reduction. 
For some large Galois extensions $L$ of $K$ containing 
all $p$-power roots of unity, 
we show the vanishing of certain Galois cohomology groups 
of $L$ with values in the $p$-adic representation associated with $E$.
We use these to prove analogous results in the
global case.
This generalizes some results of Coates, Sujatha and Wintenberger.
\end{abstract}
 

 
\section{Introduction} 

The vanishing of cohomology groups associated with $p$-adic Galois 
representations defined by elliptic curves is one of the useful
results towards generalization of methods in Iwasawa theory 
to larger Galois extensions. Such vanishing 
enables the computation of Euler characteristics for discrete modules 
associated to $p$-adic Galois representations
\cite{CSW} and Selmer groups of elliptic 
curves over extensions containing all $p$-power 
roots of unity \cite{CH}, \cite{CSS}. Our 
purpose in this paper is to show the vanishing 
of cohomology groups with values in a geometric $p$-adic Galois 
representation with respect to some large Galois extensions. 
In particular, we consider extensions of a 
$p$-adic field obtained by adjoining the coordinates 
of $p$-power torsion points on an elliptic curve.

Let $p$ be a prime number. For the moment, we let $K$ be 
any field with characteristic not equal to $p$. 
Fix a separable closure $\overline{K}$ of $K$. 
Put $G_K := \Gal(\overline{K}/K)$, the absolute Galois 
group of $K$. 
Let $X$ be a proper smooth variety defined over $K$. 
For each $i \geq 0$, we let
\begin{equation}\label{defnetale}
V = H^{i}_{\text{\'et}} (X_{\overline{K}}, \mathbb{Q}_p) 
= \varprojlim H^{i}_{\text{\'et}} (X_{\overline{K}}, \mathbb{Z} / {p^n \mathbb{Z}}) 
\otimes_{\mathbb{Z}_p} \mathbb{Q}_p
\end{equation}
denote the $i$th \'etale cohomology group of 
$X_{\overline{K}}:= X \otimes_K {\overline{K}}$ 
having coefficients in $\mathbb{Q}_p$, which is a finite-dimensional 
vector space over $\mathbb{Q}_p$. We denote by
\[ \rho : G_K \longrightarrow \GL_{\mathbb{Q}_p}(V) \]
the homomorphism giving the action of $G_K$ on the vector space $V$.

For a general finite-dimensional vector space 
$V$ over $\mathbb{Q}_p$ and a compact subgroup 
$G$ of $\GL(V)$, we write $H^i(G,V)$ ($i = 0,1,\ldots$) for the 
cohomology groups of $G$ acting on $V$ 
defined by continuous cochains, where 
$V$ is endowed with the $p$-adic topology. 
We say that $V$ has \emph{vanishing $G$-cohomology} 
if $H^i(G,V) = 0$ for all $i \geq 0$. 

For a Galois representation $(\rho,V)$ as given above, 
we denote by $K(V)$ the fixed subfield in $\overline{K}$ 
by the kernel of $\rho$.
For a subfield $L$ of $\overline{K}$, let 
$G_L$ denote the subgroup of $G_K$ corresponding to $L$.
Let $K(\mu_{p^\infty})$ be the smallest field 
extension of $K$ which contains all $p$-power roots of unity. 
Denote by $G_V$ (resp.\ $H_V$) the image of $G_K$ 
(resp.\ $G_{K(\mu_{p^\infty})}$) under $\rho$. 
We may identify $G_V$ (resp.\ $H_V$) with the Galois 
group of $K(V)$ over $K$ (resp.\ $K(V) \cap K(\mu_{p^\infty})$). 

We assume henceforth that $K$ is a finite extension 
of $\Q_p$. We recall a theorem due to Coates, Sujatha and Wintenberger.
Although the result was originally motivated by computation
of Euler characteristics associated to $V$, this theorem 
turned out to be useful in dealing with some problems 
in non-commutative Iwasawa theory (cf.\ \cite{CSS}).

\begin{thm}[\cite{CSW}, Theorems 1.1 and 1.5]\label{thm0}
Let $X$ be a proper smooth variety defined over $K$ 
with potential good reduction. Let $i$ be a positive odd integer 
and put $V = H^{i}_{\text{\rm{\'et}}} (X_{\overline{K}}, \mathbb{Q}_p)$.
Then $V$ has vanishing $G_V$-cohomology and vanishing $H_V$-cohomology.
\end{thm}

We proceed further in view of the theorem above. 
Consider an arbitrary Galois extension $L/K$ 
contained in $\overline{K}$. 
Put $J_V = \rho(G_L)$. We then ask, when do we 
obtain vanishing $J_V$-cohomology? 
Clearly if $L$ is ``too close'' to $K(V)$, 
the vanishing cannot be attained.
In consideration of the theorem above, 
we then expect that if $L$ and $K(V)$ are
sufficiently independent over the field 
$K_{\infty,V} := K(V) \cap K(\mu_{p^\infty})$ 
(or $K$), then $V$ has vanishing $J_V$-cohomology. 
This does not necessarily mean that
the intersection $M = K(V) \cap L$ is of 
finite degree over $K_{\infty, V}$ (or $K$),
although we have found no
example in which the intersection $M$ 
is of infinite degree over $K_{\infty,V}$ and
$V$ has vanishing $J_V$-cohomology.
On a related note, what can be said 
if $L$ is defined by another geometric representation?
More precisely, suppose we have another $p$-adic
Galois representation $(\rho',V')$. 
Let us put $J_V = \rho(G_{K(V')})$
and $J_{V'} = \rho'(G_{K(V)})$. 
How does the vanishing of $J_V$-cohomology of $V$ 
relate with the vanishing of $J_{V'}$-cohomology 
of $V'$?  

Some results have been 
obtained for the vanishing of $H^0 (J_V, V)$. For instance, 
Ozeki proved in \cite{Ozeki} that when $V$ is given by an abelian 
variety with good ordinary reduction, the vanishing of 
$H^0 (J_V, V)$ is equivalent to the property that the degree
of the residue field of $L$ over the residue field $k$ of $K$  
has finite $p$-part. It was further shown
that when $V$ is given by an elliptic curve and $L$ is 
the field of division points of $p$-power order of 
another elliptic curve, $H^0 (J_V, V)$ vanishes depending 
on the types of reduction of the elliptic curves involved.
In \cite{KT}, Kubo and Taguchi have shown that $H^0 (J_V, V)$ 
vanishes in the general setting where $K$ is a complete discrete
valuation field of mixed characteristic and $L$ is a subfield 
of the Kummer extension $K(\sqrt[p^\infty]{K^{\times}})$. 

After recalling some related facts, 
we prove the following result in \S 3 which provides a simple
criterion for determining the vanishing of $J_V$-cohomology
from the Lie algebras of $\Gal(K(V)/(K(V) \cap K(\mu_{p^\infty})))$ 
and $\Gal (L/(L \cap K(\mu_{p^\infty})))$. 

\begin{thm}\label{main} 
Let $X$ be a proper smooth variety over $K$ with potential
good reduction and $i$ a positive odd integer. Put 
$V = H^{i}_{\text{\rm{\'et}}} (X_{\overline{K}}, \mathbb{Q}_p)$
and $K_{\infty,V} := K(V) \cap K(\mu_{p^\infty})$.
Let $L/K$ be any $p$-adic Lie extension such 
that $K(\mu_{p^\infty})$ is of finite degree
over $K_{\infty,L} := L \cap K(\mu_{p^\infty})$. 
Assume that the Lie algebras 
\[ \Lie(\Gal(K(V)/K_{\infty,V})) \text{ and } 
\Lie (\Gal (L/K_{\infty,L}))\] have no common 
simple factor. Then $V$ has vanishing $J_V$-cohomology, 
where $J_V = \rho(G_L)$.
\end{thm}

For an elliptic curve $E$ over $K$, we denote by
\[ \rho_E : G_K \longrightarrow \GL(T_p(E)) \simeq \GL_{2}(\mathbb{Z}_p) \]
the natural continuous representation 
associated with the Tate module $T_p(E)$ of $E$. 
We use the usual notation 
$V_p(E) = T_p(E) \otimes_{\mathbb{Z}_p} \mathbb{Q}_p$. 
Let $\mathbb{Q}_p(r)$ denote the $r$th twist by the 
$p$-adic cyclotomic character, where $r \in \mathbb{Z}$.
Note that the dual $V_p(E)^{\vee} = \Hom(V_p(E),\mathbb{Q}_p)$ is 
canonically isomorphic to $H^{1}_{\text{\'et}} (E_{\bar{K}}, \mathbb{Q}_p)$.  
On the other hand, the Weil pairing allows us
to identify $V_p(E)$ with $V_p(E^\vee)$ in 
a canonical way.
Thus we may canonically identify $V_p(E)$ with
$H^{1}_{\text{\'et}} (E^{\vee}_{\bar{K}}, \mathbb{Q}_p(1))$.
We also note that $K(V_p(E)) = K(E_\infty)$, 
where $K(E_\infty)$ is the extension of $K$ 
generated by the coordinates of all the $p$-power 
torsion points on the group of $\overline{K}$-valued
points $E(\overline{K})$. By the Weil pairing, 
the field $K(E_\infty)$ contains $K(\mu_{p^\infty})$.  

In \S \ref{sec:ord} we consider the setting where
the variety in question has good ordinary reduction.
For elliptic curves with good ordinary reduction over 
$K$, we obtain a necessary and sufficient condition
on $L$ so that the $p$-adic Galois representation 
$V=V_p(E)$ has vanishing $J_V$-cohomology. 

\begin{thm}(see Theorem \ref{thm2A} 
and Corollary \ref{van-ord-abelvar})\label{thmord} 
Let $E$ be an elliptic curve with potential 
good ordinary reduction over $K$ and $L$ be a 
Galois extension of $K$.
Put $V=V_p(E)$ and $J_V = \rho_E(G_L)$.  \\
(1) If the residue field $k_L$ of 
$L$ is a potential prime-to-$p$ extension
of the residue field $k$ of $K$
(in the sense of Definition \ref{prime-to-p}),
then $V$ has vanishing $J_V$-cohomology. \\
(2) Assume that $E$ has good ordinary reduction 
over $K$ and that $L$ contains $K(\mu_{p^\infty})$ 
and the coordinates of the $p$-torsion points of $E$.
Then $V$ has vanishing $J_V$-cohomology if and 
only if $k_L$ is a potential prime-to-$p$ 
extension over $k$.
\end{thm}

Now we consider the case where $L=K(V')$ is 
given by another ``geometric" representation $V'$. 
In this case, we often encounter that $V'$
also has vanishing $J_{V'}$-cohomology, where
$J_{V'}=\rho'(G_{K(V)})$.
This motivates us to introduce the 
notion of ``cohomological coprimality" in 
\S \ref{cohom_coprime}.

\begin{thm}(see Theorem \ref{thm3})\label{thm} 
Let $X$ be a proper smooth variety over $K$ with potential
good ordinary reduction (in the sense of Definition \ref{BK}) 
and $i$ a positive odd integer. Let $E/K$ be an elliptic curve  
with potential good supersingular reduction. 
Put $V = H^{i}_{\text{\rm{\'et}}} (X_{\overline{K}}, \mathbb{Q}_p)$
and $V' = V_p(E)$. Then $V$ and $V'$ are cohomologically coprime.
\end{thm}

Suppose $E$ and $E'$ are elliptic curves over $K$.
In \S \ref{sec:vanish for EC}, we prove some
results on the cohomological coprimality of
$V_p(E)$ and $V_p(E')$. This is done by
distinguishing the reduction types of $E$ and $E'$.
We summarize our results in the following theorem.
We also note that its proof provides extensions 
to some of the results obtained in \cite{Ozeki}.

\begin{thm} 
Let $E$ and $E'$ be elliptic curves over $K$. 
The cohomological coprimality of  
$V_p(E)$ and $V_p(E')$ is given by the following table:

\begin{center}
\begin{tabular}{|c|c|c|}
\hline
$E$ & $E'$ & Cohomologically coprime \\ 
\hline
\multirow{3}{3cm}{\centering ordinary} & 
	\multirow{1}{5cm}
		{\centering ordinary} & No$^{\sharp}$ \\
	\cline{2-3}
		{} & supersingular & Yes \\
	\cline{2-3}
		{} & multiplicative & ``No" \\	
\hline
\multirow{3}{3cm}{\centering supersingular \\ with FCM} & \multirow{1}{5cm}{\centering supersingular with FCM} & Yes$^{\ast}$ \\
\cline{2-3}
{} & \multirow{1}{5cm}{\centering supersingular without FCM} & Yes \\
\cline{2-3}
{} & \multirow{1}{5cm}{\centering multiplicative} & ``No" \\
\hline
\multirow{2}{3cm}{\centering supersingular \\ without FCM} & \multirow{1}{5cm}{\centering supersingular without FCM} & Yes$^{\ast}$ \\
\cline{2-3}
{} & \multirow{1}{5cm}{\centering multiplicative} & ``No" \\
\hline
multiplicative & multiplicative & ``No"\\
\hline
\end{tabular}
\end{center}
\end{thm}
In the table above, FCM means formal complex multiplication. 
The symbol $^{\ast}$ means conditional cohomological
coprimality. The cohomological coprimality in 
this case holds under the additional assumption that 
the group $E(K(E_\infty))[p^\infty]$ of 
$K(E_\infty)$-rational points of 
$E$ of $p$-power order is finite. For $\sharp$, refer 
to Remark \ref{rem3}. For the case where one of the
elliptic curves has multiplicative reduction, we refer 
to Remark \ref{cohom-coprime-mult}.
The rest is provided by Theorem \ref{thm6} 
in \S\ref{subsec:good reduction}. 

In the last section we use the local results above to
prove cohomological coprimality results associated
with ``global" $p$-adic Galois representations.
Indeed, the proof of Theorem \ref{thm0} in \cite{CSW} 
relies on showing the existence of some special operator in the
Lie algebra 
$\mathfrak{h} = \Lie(\rho(G_{K(\mu_{p^\infty})})) \otimes_{\Q_p} \overline{\Q_p}$. 
Such operator satisfies a criterion introduced by
Serre (cf.\ \cite{Ser3}) for the vanishing of Lie algebra cohomology groups,
and hence Theorem \ref{thm0} follows by a 
well-known result due to Lazard.
In general, enlarging the field $L$ makes 
the corresponding image group $J_V$ 
smaller and hence the special operator of
$\mathfrak{h}$ may no longer lie in the 
Lie algebra 
$\mathfrak{j}=\Lie(J_V) \otimes_{\Q_p} \overline{\Q_p}$. 
Our methods in the local case ensure 
that such operator still belongs to $\mathfrak{j}$
which in turn belongs to the corresponding Lie algebra of the
image of the global Galois group.  

The proof of Theorem \ref{thm0} as described above 
also implies that $V$ has vanishing $G_V$-cohomology 
if and only if $V$ has vanishing $G$-cohomology for 
an open subgroup $G$ of $G_V$. The same statement holds for $H_V$.  
In this paper Galois representations are defined by
objects which satisfy a ``potential condition over $K$", 
in the sense that the given condition is satisfied
after a suitable finite extension $K'$ of $K$.
In most of our proofs we will often encounter the phrase 
``replacing $K$ by a finite extension" in order
to reduce the proof to a more convenient
setting, e.g.\ so that the varieties in question 
have good reduction over $K$ and that $K(V)$ 
contains all the $p$-power roots
of unity. By the above remark, 
there is no harm in doing this kind of reduction.

\noindent \emph{Acknowledgements.} 
The author would like to express his 
sincere gratitude to Yuichiro Taguchi,
who proposed the theme of this paper, for his patience, 
helpful advice and contributions to this paper. 
He would also like to thank the Ministry 
of Education, Culture, Sports, Science and Technology of Japan 
for its support. He is currently on study leave from the Institute 
of Mathematics, University of the Philippines Diliman.

\section{Preliminaries and review of some known results}

\subsection{Vanishing Cohomology and Cohomological Coprimality}\label{cohom_coprime} 

Let $G$ be a topological group and $F$ be a topological field. 
A continuous $F$-linear representation of $G$
is a finite-dimensional $F$-vector space $V$ equipped with a 
continuous linear action of $G$. Equivalently, 
it is a continuous homomorphism $\rho: G \rightarrow \GL_{F}(V)$.
We denote a continuous $F$-linear representation 
of $G$ by $(\rho,V)$. 
For a prime number $p$,
a continuous $\Q_{p}$-linear representation of $G$
is called a $p$-adic representation of $G$.

For a general finite-dimensional $F$-vector space $V$
and a compact subgroup $\mathcal{G}$ of $\GL_{F}(V)$, 
we consider the cohomology groups $H^i(\mathcal{G},V)$
($i = 0, 1, \ldots$) of $\mathcal{G}$ acting on $V$ 
which are defined by continuous cochains 
(where $V$ is endowed with the topology induced by $F$).
\begin{defn}\label{van-cohom} 
\normalfont
For $V$ and $\mathcal{G}$ as above, we say that $V$ 
has \emph{vanishing $\mathcal{G}$-cohomology} 
if $H^i(\mathcal{G},V) = 0$ for all $i \geq 0$.
\end{defn}
 
\begin{defn}\label{defn_cohom_coprime}		 
\normalfont
Let $(\rho,V)$ and $(\rho',V')$ be two 
continuous $F$-linear representations of $G$.
Put $\mathcal{G} = \rho(\Ker \rho')$ and  
$\mathcal{G}' = \rho'(\Ker \rho)$.
We say that $V$ and $V'$ are
\emph{cohomologically coprime} 
if $V$ has vanishing $\mathcal{G}$-cohomology and 
$V'$ has vanishing $\mathcal{G}'$-cohomology. 
\end{defn}

\begin{remark}
\normalfont
In fact, the above notion of cohomological coprimality 
can be formulated for topological modules
over a topological ring with a continuous 
action by a fixed topological group $G$. Moreover, 
two topological modules (with a continuous action by
$G$) being compared against need not 
have the same ring of scalars. For our purposes, 
we restrict our definition in the above form.
\end{remark}

\subsection{The setup}\label{setup}

Let $K$ be a finite extension of $\Q_p$.
For a $p$-adic representation 
$(\rho,V)$ of $G_K$, 
we denote by $K(V)$ the fixed subfield of 
$\overline{K}$ by the kernel of $\rho$. 
Let $\mu_{m}$ denote the group of $m$th roots
of unity for $m \in \Z$ with $m \geq 1$. 
Denote by $\mu_{p^\infty}$ the
union of all $\mu_{p^n}$ as $n$ runs over the 
set of all positive integers.
We identify $G_V = \rho(G_K)$ with the Galois group 
$\Gal(K(V)/K)$. 
For an arbitrary Galois extension $L/K$,
we may identify $J_V = \rho(G_L)$ with a 
closed subgroup of $G_V$, whose fixed field 
$M = K(V)^{J_V}$ is the intersection 
of $K(V)$ and $L$. Then the Galois group 
$\Gal(M/K)$ may be identified with a quotient of 
$G_V$. This latter group is a $p$-adic Lie group and 
thus, so is $\Gal(M/K)$ (cf.\ \cite{DSMS}, Theorem 9.6 (ii)).
Hence, $M/K$ is a $p$-adic Lie extension. 
If $L = K(\mu_{p^\infty})$, we write $H_V$ 
for $\rho(G_L)$ instead of $J_V$ and 
$M$ in this case will be written as 
$K_{\infty,V} := K(V) \cap K(\mu_{p^\infty})$.
When $V = V_p(E) = T_p(E) \otimes_{\Z_p} \Q_p$ is 
given by an elliptic curve $E$ over $K$,
we write $K(E_\infty)$ instead of $K(V)$.
In this case we have $K_{\infty,V} = K(\mu_{p^\infty})$ 
by the Weil pairing.

We have the following diagram of fields:
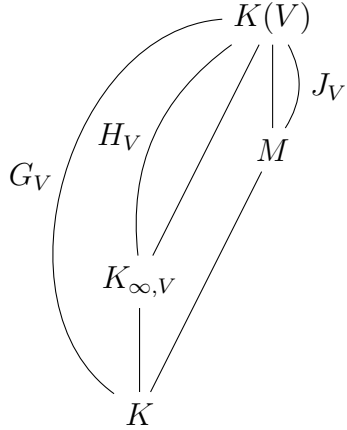
\begin{center}
\begin{tikzpicture}[%
  back line/.style={densely dotted},
  cross line/.style={preaction={draw=white, -,line width=6pt}}]
  \node (A) {$K(V)$};  
  \node (B) [below of=A, node distance=1.75cm] {$M$};
  \node [below of=A, node distance=1.5cm, left of=B, node distance=1.75cm] (C) {$K_{\infty,V}$};
  \node [below of=C, node distance=1.75cm] (D) {$K$};
  
  \draw[cross line] (A) -- (B)  -- (D);  
  \draw[cross line] (A) -- (C)  -- (D);	  
  \draw[-, bend right] (A) to node [left]{$H_V$} (C);
  \draw[-, bend left] (A) to node [right]{$J_V$} (B);
  \draw[-, bend right = 70pt] (A) to node [left]{$G_V$} (D); 
  \end{tikzpicture}
\end{center}

In particular, if $L$ contains $K(\mu_{p^\infty})$, 
then $J_V$ is a closed normal subgroup of $H_V$ 
and $M$ contains $K_{\infty, V}$.

\subsection{Some lemmas}
Let $L$ be a Galois extension of $K$ which contains
$K(\mu_{p^\infty})$. Put $\mathscr{G} = \Gal(L/K)$ and 
$\mathscr{H} = \Gal(L/K(\mu_{p^\infty}))$.
Let $\varepsilon : \mathscr{G} \rightarrow \Z_p^{\times}$ 
be a continuous character of $\mathscr{G}$ 
whose image is open in $\Z_p^{\times}$.
The group $\mathscr{G}$ acts on $\mathscr{H}$ by inner automorphisms, that is, for 
$\sigma \in \mathscr{G}$ and $\tau \in \mathscr{H}$, we have 
$\sigma \cdot \tau = \sigma \tau \sigma^{-1}$. 
Assume that the following relation holds:
\begin{equation}\label{eqn1}
\sigma \cdot \tau = \tau^{\varepsilon(\sigma)} 
\end{equation} 
for all $\sigma \in \mathscr{G}$, $\tau \in \mathscr{H}$.

\begin{lemma}\label{lemB}
Let $(\psi, W)$ be a $p$-adic representation of $\mathscr{G}$.
Let $\varepsilon$ be a character as above
and suppose the action of $\mathscr{G}$
on $\mathscr{H}$ satisfies relation (\ref{eqn1}).
Then after a finite extension $K'/K$, the subgroup
$\mathscr{H}$ acts unipotently on $W$.
\end{lemma}

\begin{proof} 
Put $d = \dim_{\Q_p} W$.
The result is trivial if $d = 0$. We assume
henceforth that $d$ is nonzero.
We may argue in the same manner as 
the proof of Lemma 2.2 of \cite{KT}.
Let $\tau \in \mathscr{H}$ and 
$\lambda_1, \ldots, \lambda_d$ be the
eigenvalues of $\psi(\tau)$.
Then relation (\ref{eqn1}) shows that
\[ \{ \lambda_1, \ldots, \lambda_d \}
= \{ \lambda_1^{\varepsilon(\sigma)}, \ldots, \lambda_d^{\varepsilon(\sigma)} \} \]
for all $\sigma \in \mathscr{G}$.
Let $e$ be a positive integer such that
$1 + p^e$ lies in $\varepsilon(\mathscr{G})$. 
Such an integer exists since 
$\varepsilon(\mathscr{G})$ is open in 
$\mathbb{Z}_p^\times$.
For each $i = 1, \ldots, d$, there exists
an integer $r_i$ with $1 \leq r_i \leq d$
such that $\lambda_i^{(1+p^e)^{r_i}} = \lambda_i$.
We then put 
\[ m = \text{LCM} \{ (1 + p^e)^r - 1| r = 1,\cdots,d \}. \]
With this choice of $m$ we see that 
$\psi(\tau)^m$ is unipotent
since $\lambda_i^m = 1$ for all $i = 1, \ldots, d$.
Hence $\mathscr{H}^m = \{ \tau^m | \tau \in \mathscr{H} \}$ 
acts unipotently on $W$. 
Then the semisimplification
of the restriction $\psi|_{\mathscr{H}}$ to $\mathscr{H}$ is
a sum of characters $\mathscr{H}/\mathscr{H}^m \rightarrow \mu_m$, 
after a suitable extension of scalars. 
These characters become trivial upon replacing
$K(\mu_{p^\infty})$ by a finite extension, say $K_{\mu}$.
In fact, $K_{\mu} = K'(\mu_{p^\infty})$ for some
finite extension $K'$ of $K$. 
\end{proof}

In this paper, $\chi : \mathscr{G} \rightarrow \mathbb{Z}_p^{\times}$
always denotes the $p$-adic cyclotomic character
(i.e.\ the continuous character such that 
$g(\zeta) = \zeta^{\chi(g)}$ 
for all $g \in \mathscr{G}$ and all 
$\zeta \in \mu_{p^\infty}$).
The above lemma will be used later in \S 4
in the case where $\varepsilon$ is the 
product of $\chi$ with another 
continuous character of $\mathscr{G}$.

\begin{lemma}\label{lemC}
Let $\varphi : \mathscr{U} \rightarrow \GL_{\Q_p}(W)$ be 
a representation of a group $\mathscr{U}$ 
on a finite-dimensional $\Q_p$-vector space $W$.
Suppose $\mathscr{U}$ acts unipotently on $W$.
Then $W^\mathscr{U} = 0$ if and only if $W = 0$.
\end{lemma}

\begin{proof}
One implication is trivial. We assume that $W^\mathscr{U} = 0$
and let $w \in W$. Since the action of $\mathscr{U}$ on $W$ is unipotent,
for all $u \in \mathscr{U}$ the element $\varphi(u) - 1$ is
nilpotent. Thus, there exists $n \in \Z_{\geq1}$ such 
that $(\varphi(u)-1)^n(w) = 0$ but 
$(\varphi(u)-1)^{n-1}(w) \neq 0$.
If $n=1$, then $\varphi(u)(w) = w$. So $w = 0$ by hypothesis.
If $n>1$, put $w' = (\varphi(u)-1)^{n-1}(w)$. Note that
$w' \neq 0$. But then we have
$(\varphi(u)-1)(w') = (\varphi(u)-1)^n(w) = 0$. 
Thus $w' \in W^{\mathscr{U}} = 0$. 
Therefore we must have $w=0$.
\end{proof}

\subsection{Ordinary Representations}

\begin{defn}\label{ordrepn} 
\normalfont 
A $p$-adic Galois representation $V$ of $G_K$ is said to be 
\emph{ordinary} if there exists a filtration by $G_K$-invariant subspaces 
$\{ \Fil^i V \}_{i \in \mathbb{Z}}$ with the following properties:
\begin{center}
$\Fil^{i+1} V \subseteq \Fil^i V$ for all $i$, \\
$\Fil^i V = V$ for $i \ll 0$ and
$\Fil^i V = 0$ for $i \gg 0$, 
\end{center}
such that the inertia subgroup $I_K$ acts on the $i$th graded quotient
$\Fil^i V / \Fil^{i+1} V$ by the $i$th power of the $p$-adic cyclotomic
character. 
\end{defn}

\begin{defn}\label{BK} 
\normalfont
We say that a proper smooth variety $X$ over $K$ has 
\emph{good ordinary reduction} over $K$ if there exists a 
smooth proper model $\mathfrak{X}$ over $\mathcal{O}_K$ 
with special fiber $\mathcal{Y}$ such that the de Rham-Witt cohomology groups 
$H^r(\mathcal{Y}, d\Omega_\mathcal{Y}^s)$ are trivial for all $r$ and all $s$. 
We say $X$ has \emph{potential good ordinary reduction} over $K$ if
it has good ordinary reduction after a finite extension $K'/K$. 
\end{defn}
In the above definition, $d\Omega_\mathcal{Y}^s$ is the sheaf 
of exact differentials on $\mathcal{Y}$. 
This definition is due to Bloch-Kato (cf.\ \cite{BK}, Definition 7.2). 
Equivalent formulations for this definition are given in Proposition 7.3 
of \emph{op.\ cit.} When $X$ is an abelian variety of dimension $g$, 
this definition coincides with the property that the 
group of $\bar{k}$-points of $\mathcal{Y}$ killed by $p$ is isomorphic to 
$\left( \mathbb{Z}/{p \mathbb{Z}} \right)^g$, which is the 
classical definition of an abelian variety with good ordinary reduction. 
Here, $\bar{k}$ denotes an algebraic closure of the residue field
$k$ of $K$.

The $\text{\'e}$tale cohomology groups of 
a proper smooth variety with good ordinary 
reduction can be characterized by the
following result of Illusie:
\begin{thm}[\cite{LI}, Cor. 2.7]
Let $X$ be a proper smooth variety over $K$ which has 
good ordinary reduction over $K$. Then 
the $\text{\'e}$tale cohomology group 
$H^{i}_{\text{\rm{\'et}}} (X_{\overline{K}}, \mathbb{Q}_p)$ ($i \geq 0$) 
is an ordinary representation of $G_K$. 
\end{thm}

\subsection{Lie algebras associated with elliptic curves}

Consider an elliptic curve $E$ over $K$. 
The structure of the Lie algebras associated to $E$ is well-known 
(cf.\ \cite{Ser1}, Appendix of Chapter IV):

\begin{prop}\label{prop3}
Let $E$ be an elliptic curve over $K$. 
Let $\mathfrak{g}:= \Lie(\rho_E(G_K))$ and 
$\mathfrak{i}:= \Lie(\rho_E(I_K))$ be 
the Lie algebras of the image of $G_K$ and 
its inertia subgroup $I_K$ under $\rho_E$, respectively 
(These are Lie subalgebras of $\End (V_p(E))$).
\begin{enumerate}
	\item[(i)] If $E$ has good supersingular reduction 
	with formal complex multiplication, then $\mathfrak{g}$ 
	is a non-split Cartan subalgebra of $\End(V_p(E))$
	and $\mathfrak{i} = \mathfrak{g}$. We have  
	$\dim \mathfrak{g} = \dim \mathfrak{i}=2$. 
	\item[(ii)] If $E$ has good supersingular reduction 
	without formal complex multiplication, then 
	$\mathfrak{g} = \End(V_p(E))$ and 
	$\mathfrak{i} = \mathfrak{g}$. We have  
	$\dim \mathfrak{g} = \dim \mathfrak{i} = 4$.
	\item[(iii)] If $E$ has good ordinary reduction 
	with complex multiplication, then $\mathfrak{g}$ 
	is a split Cartan subalgebra of $\End(V_p(E))$. 
	We have $\dim \mathfrak{g} = 2$ and $\mathfrak{i}$ 
	is a $1$-dimensional subspace of $\mathfrak{g}$. 
	\item[(iv)] If $E$ has good ordinary reduction 
	without complex multiplication, then $\mathfrak{g}$ 
	is the Borel subalgebra of $\End(V_p(E))$ which 
	corresponds to the kernel of the reduction map 
	$V_p(E) \rightarrow V_p(\widetilde{E})$. 
	We have $\dim \mathfrak{g} = 3$ and 
	$\mathfrak{i}$ is a $2$-dimensional subspace of 
	$\mathfrak{g}$ with $\mathfrak{i}/[\mathfrak{i},\mathfrak{i}]$ 
	of dimension $1$.
	\item[(v)] If $E$ has $j$-invariant with negative $p$-adic valuation, 
	then $\mathfrak{g}$ is the subalgebra of $\End(V_p(E))$ 
	which consists of the endomorphisms $u$ for which $u(V_p(E)) \subset W$,
	where $W$ is the unique $G_K$-stable $1$-dimensional subspace of $V_p(E)$. 
	Moreover, $\mathfrak{i} = \mathfrak{g}$. We have 
	$\dim \mathfrak{g} = \dim \mathfrak{i} = 2$

\end{enumerate}   
\end{prop}

Denote the ring of integers of $K$ by $\mathcal{O}_K$. 
In the proposition above, an elliptic curve $E$ over $K$ 
with good supersingular reduction is said to have 
\emph{formal complex multiplication over $K$} if 
the endomorphism ring of the $p$-divisible group 
$\mathcal{E}(p)$ associated with the N{\'e}ron model 
$\mathcal{E}$ of $E$ over $\mathcal{O}_K$ is a 
$\mathbb{Z}_p$-module of rank $2$. 
We simply say $E$ has \emph{formal complex multiplication} 
if $E \times_K K'$ has formal complex multiplication for 
some algebraic extension $K'$ of $K$. 
Then the quadratic field 
$\text{End}_{\mathcal{O}_{K'}} (\mathcal{E}(p)) \otimes_{\mathbb{Z}_p} \mathbb{Q}_p$ 
is called the \emph{formal complex multiplication field} of $E$. 
We can take for $K'$ a finite extension of $K$ of degree at most $2$.
The subspace $W$ in \emph{(v)} of the proposition 
is isomorphic to the twist of $\Q_p(1)$ by an unramified character
of order at most $2$. 

\section{Some Criteria for the Vanishing of $J_V$-cohomology} \label{sec: vanishing} 

Let $X$ be a proper smooth variety over $K$ 
with potential good reduction. Let $i$ be a positive odd integer and put 
$V = H^{i}_{\text{\'et}} (X_{\overline{K}}, \mathbb{Q}_p)$. Let $\rho$
be the continuous homomorphism attached to $V$ as in the Introduction
and let $G_V = \rho(G_K)$ and $H_V = \rho (G_{K(\mu_{p^\infty})})$.
In this section, we prove Theorem \ref{main}
given in the Introduction. 
Before embarking on the proofs, we give a few remarks.

Recall from \S\ref{setup} that we have the isomorhisms $G_V \simeq \Gal(K(V)/K)$ and 
$H_V \simeq \Gal(K(V)/K_{\infty,V})$.
If $L/K$ is a Galois extension of $K$,
then we identify $J_V := \rho(G_L) \simeq \Gal(K(V)/M)$ 
where $M = K(V) \cap L$. 

\begin{lemma}\label{lemma6} 
Let $V$, $G_V$, $H_V$ and $J_V$ be as above. 
\\
(1) If $J_V$ has finite index in $G_V$, 
then $V$ has vanishing $J_V$-cohomology. \\
(2) If $L$ contains $K(\mu_{p^\infty})$ and 
$J_V$ has finite index in $H_V$,
then $V$ has vanishing $J_V$-cohomology.
\end{lemma}

\begin{remark}\label{rem1}
\normalfont
By Galois theory, we have 
$\Gal(M/K) \simeq G_V/J_V$ 
(resp.\ $\Gal(M/K_{\infty,V}) \simeq H_V/J_V$) 
in the discussion above.
So the condition that $J_V$ has finite index in 
$G_V$ (resp.\ $H_V$) is equivalent to the finiteness 
of the degree of the extension $M$ over $K$ 
(resp.\ $K_{\infty,V}$).
\end{remark}

\begin{proof}
Replacing $K$ with a finite extension, we may assume
$G_V = J_V$ (resp.\ $H_V = J_V$). 
It follows immediately from Theorem \ref{thm0} 
that $V$ has vanishing $J_V$-cohomology.
\end{proof}

\begin{remark} \label{rem2}
\normalfont
Let $V$ be any $p$-adic representation of $G_K$ as above and 
$L/K$ a Galois extension containing 
$K(\mu_{p^\infty})$. Take a $G_K$-stable 
$\mathbb{Z}_p$-lattice $T$ of $V$. 
It is known (cf.\ e.g.\ \cite{KT}, Lemma 2.1) that 
the vanishing of $H^0(J_V, V)$, with $J_V = \rho(G_L)$,
is equivalent to the finiteness of $(V/T)^{G_L}$.
Hence, since $V$ has vanishing $H_V$-cohomology,
we have the relation 
$(1)\Rightarrow(2)\Rightarrow(3)$
between the following statements:
\\
(1) 
$M$	is a finite extension
of $K(\mu_{p^\infty})$,
\\
(2) 
$V$ has vanishing $J_V$-cohomology, and
\\
(3) 
$(V/T)^{G_L}$ is a finite group.
\\
However, converses may not necessarily hold.
In some cases though, we have $(3)\Rightarrow(1)$, 
as we shall see in Corollary \ref{cor_ss}.
\end{remark}

To prove Theorem \ref{main} given in the Introduction,
we need the following lemma.

\begin{lemma}\label{lemma1} 
Let $X$ be a proper smooth variety over a $p$-adic field 
$K$ with potential good reduction and $i$ be 
a positive integer. Consider the representation
$(\rho, V)$, where 
$V = H^i_{\text{\rm{\'et}}}(X_{\bar{K}}, \mathbb{Q}_p)$ and
let $\det \rho: G_K \rightarrow \mathbb{Z}_p^{\times}$
be the character obtained by 
composing $\rho$ with the determinant map.
Then $\det \rho = \chi^{-\frac{id}{2}}$ on an open
subgroup of $G_K$, where $d = \dim_{\Q_p} V$
and $\chi$ is the $p$-adic cyclotomic character.
\end{lemma}

\begin{proof}
Replacing $K$ by a finite extension, 
we may assume $V$ is crystalline. 
Consider the filtered $\varphi$-module 
$D_{\mathrm{cris}}(V)$ and let $\Phi = \varphi^f$, 
where $q=p^f$ is the cardinality
of the residue field of $\mathcal{O}_K$.
Let $\delta$ denote the determinant 
of the endomorphism $\Phi$.
By \cite{C-LS}, the characteristic polynomial
of $\Phi$ has rational coefficients and 
its roots are Weil numbers of weight $i$.  
Thus, in particular, $\delta$ is a rational number
and it has archimedean absolute value
equal to $q^{t}$ where $t=id/2$. 
Hence $\delta = \pm q^t$.
Since $\det \rho$ is crystalline,           
the restriction of the 
character $\det \rho$ to $I_K$ is
equal to $\chi^{-t}$ (cf.\ \cite{Fon}, Proposition 5.4.1).
Thus, $\det \rho = \eta \chi^{-t}$
with $\eta$ an unramified character.
By Lemma 3.4 in \cite{CSW},
the character $\eta$ has order at most two,
from which the desired result follows.
Note that the Betti number $d$ is even if $i$ 
is odd, by the Hodge symmetry.
\end{proof}

\begin{proof}[(Proof of Theorem \ref{main})]   
The theorem clearly holds if $d = \dim_{\mathbb{Q}_p} V$ is zero.
We assume henceforth that $V$ is of positive dimension.
Since the kernel of $\rho$ is contained in the 
kernel of $\det \rho$, we see that $K(V)$ contains the fixed subfield 
$K(\det V)$ of $\overline{K}$ by the kernel of $\det \rho$.
Note that the character $\det \rho$ is the $-id/2$-th power of the
$p$-adic cyclotomic character on an open subgroup of 
$G_K$, by Lemma \ref{lemma1}. Hence the field 
$K(V)$ contains a subfield $F$ of $K(\mu_{p^\infty})$ such that
$K(\mu_{p^\infty})$ is of finite degree over $F$ since $d > 0$. 
Replacing $K$ by a finite extension, we may then assume that 
$K(V)$ and $L$ contains $K(\mu_{p^\infty})$. 
Put $\mathfrak{h} = \Lie(\Gal(K(V)/K(\mu_{p^\infty})))$ and 
$\mathfrak{h}' = \Lie(\Gal(L/K(\mu_{p^\infty})))$. 
Recall that $M$ is the intersection of the fields 
$K(V)$ and $L$, which is a Galois extension of 
$K(\mu_{p^\infty})$.
Let $\mathfrak{j}$ and $\mathfrak{j}'$ be the Lie algebras 
of $\Gal(K(V)/M)$ and $\Gal(L/M)$ respectively. 
The Lie algebra $\mathfrak{j}$ (resp. $\mathfrak{j}'$) is
an ideal of $\mathfrak{h}$ (resp. $\mathfrak{h}'$), 
since $\Gal(K(V)/M)$ (resp. $\Gal(L/M)$) is a 
closed normal subgroup of
$\Gal(K(V)/K(\mu_{p^\infty}))$ 
(resp. $\Gal(L/K(\mu_{p^\infty}))$).
We have
\[ \frac{\mathfrak{h}}{\mathfrak{j}} \simeq 
\Lie\left( \frac{\Gal(K(V)/K(\mu_{p^\infty}))}{\Gal(K(V)/M)} \right) \simeq 
\Lie(\Gal(M/K(\mu_{p^\infty}))) \simeq 
\Lie\left( \frac{\Gal(L/K(\mu_{p^\infty}))}{\Gal(L/M)} \right) \simeq 
\frac{\mathfrak{h}'}{\mathfrak{j}'}. \]
The above expressions are all equal to zero by hypothesis.
Therefore $\Gal(K(V)/M)$ has finite index in $\Gal(K(V)/K(\mu_{p^\infty}))$.
We then apply Lemma \ref{lemma6} to obtain the desired conclusion.
\end{proof}

It seems worthwhile to state the following corollaries 
for cohomological coprimality in
the case where $L$ is given by another ``geometric" representation. 
More precisely, consider another proper smooth variety $Y$ over $K$ 
with potential good reduction. Let $j$ be a positive odd integer and put 
$V_1 = V$, as above and 
$V_2 = H^{j}_{\text{\'et}} (Y_{\overline{K}},\mathbb{Q}_p)$.
Put $J_1 = \rho_1(\Ker(\rho_2))$ and $J_2 = \rho_2(\Ker(\rho_1))$. 
Note that $J_r$ is a closed normal subgroup 
of $H_r = \rho_r (G_{K(\mu_{p^\infty})})$
(after a finite extension) for $r=1,2$. 
We have the following special case of Lemma \ref{lemma6}.

\begin{cor}\label{cor1} 
Let $V_1$ and $V_2$ be as above and let
$K(V_1)$ and $K(V_2)$ be the fixed fields
of $\Ker(\rho_1)$ and $\Ker(\rho_2)$, respectively.
If $M := K(V_1) \cap K(V_2)$ 
is a finite extension of $M \cap K(\mu_{p^\infty})$, 
then $V_1$ and $V_2$ are cohomologically coprime. 
\end{cor}

The cohomological coprimality can also be derived by 
comparing the Lie algebras $\mathfrak{h}_1 = \Lie(H_1)$ and 
$\mathfrak{h}_2 = \Lie(H_2)$.

\begin{cor}\label{prop2} 
With the assumptions and notations in the discussion above,
suppose $\mathfrak{h}_1$ and $\mathfrak{h}_2$ 
have no common simple factor. Then $V_1$ and $V_2$
are cohomologically coprime.
\end{cor}

\begin{proof}
Apply Theorem \ref{main} with $V = V_1$ 
and $L = K(V_2)$;   
and with $V = V_2$ and $L = K(V_1)$.
\end{proof}

\section{The ordinary case}\label{sec:ord}

We use the notation as in the previous 
sections. In this section, we look at the 
vanishing of cohomology groups for $p$-adic 
Galois representations given by varieties
with good ordinary reduction. 
We begin with a definition and a few remarks.

\begin{defn}\label{prime-to-p}
\normalfont
Let $F$ be a field. For an algebraic extension $F'$ 
of $F$, we say that $F'$ is a \emph{prime-to-$p$ extension}
of $F$ if $F'$ is a union of finite extensions over $F$
of degree prime-to-$p$. If $F'$ is a prime-to-$p$ extension
over some finite extension field of $F$, we say that
$F'$ is a \emph{potential prime-to-$p$ extension} of $F$.
\end{defn}

\begin{remark}\label{prime-to-p-remark}
\normalfont
(i) 
Clearly, if $F'$ is a potential prime-to-$p$
extension of $F$, then every intermediate field $F''$
(with $F \subseteq F'' \subseteq F'$) is a potential
prime-to-$p$ extension of $F$. \\
(ii)
Let $L$ be a $p$-adic Lie extension of $K$ containing 
$K(\mu_{p^\infty})$.  
Then the residue field $k_L$ is a potential 
prime-to-$p$ extension over $k$ if and only 
if $k_L/k$ is a finite extension.
\end{remark}

We now give the first main result in this section.
We consider the case given by elliptic
curves. 

\begin{thm}\label{thm2A}
Let $E/K$ be an elliptic curve with potential good 
ordinary reduction over $K$. 
Let $L$ be a Galois extension of $K$ 
whose residue field $k_L$ is a potential 
prime-to-$p$ extension over $k$. 
Put $V =V_p(E)$ and $J_V = \rho_E(G_L)$. 
Then $V$ has vanishing $J_V$-cohomology.
\end{thm}

As a corollary, we obtain necessary and 
sufficient conditions for the vanishing of 
$J_V$-cohomology groups for 
$p$-adic representations given by an elliptic 
curve with good ordinary reduction over $K$.
Let $\tilde{E}$ denote the reduction of
$E$ modulo the maximal ideal of $\mathcal{O}_K$.

\begin{cor}\label{van-ord-abelvar}
Let $E$ be an elliptic curve over $K$ with good ordinary 
reduction and $L$ be a Galois extension with residue field $k_L$. 
Assume that $L$ contains $K(\mu_{p^\infty})$ and 
the coordinates of the $p$-torsion points of $E$. 
Put $V=V_p(E)$ and $J_V=\rho_E(G_L)$. 
Then the following statements are equivalent:\\
(1) $E(L)[p^\infty]$ is finite, \\
(2) $E^{\vee}(L)[p^\infty]$ is finite, \\
(3) $\tilde{E}(k_L)[p^\infty]$ is finite, \\
(4) $\tilde{E}^{\vee}(k_L)[p^\infty]$ is finite, \\
(5) $k_L$ is a potential prime-to-$p$ extension of $k$\\
(6) $V$ has vanishing $J_V$-cohomology.
\end{cor}

\begin{proof}
The equivalence of the first five statements 
is given by Corollary 2.1 in \cite{Ozeki}.
Theorem \ref{thm2A} shows that condition (5) 
implies condition (6).  
Note that condition (1) is equivalent to 
$H^0(J_V, V) = 0$ (cf.\ e.g.\ \cite{KT}, Lem. 2.1),
so condition (6) implies (1).
\end{proof}



We now give the proof of Theorem \ref{thm2A}. 
First we note that we may reduce the proof to the case 
$L=L(\mu_{p^\infty})$. Indeed letting
$L'=L(\mu_{p^\infty})$ and $J'_V = \rho(G_{L'})$,
then $J_V'$ is a closed normal subgroup of $J_V$
and we see that if $V$ has vanishing $J'_V$-cohomology
then
\[ H^n (J_V, V) \simeq H^n(J_V/{J'_V}, H^0(J_V',V)) \hspace{20pt} n \geq 0,\]
and a priori, $V$ has vanishing $J_V$-cohomology.
We assume henceforth that $L=L(\mu_{p^\infty})$.
Considering $H^n(J_V,V)$ as a representation of $H_V/J_V$, 
we see that Theorem \ref{thm2A} follows if we prove 
the following lemma.

\begin{lemma}\label{lemA}
Assume the hypothesis in Theorem \ref{thm2A}. 
Then after a finite extension $K'/K$,
the quotient $H_V / J_V$ acts unipotently 
on $H^n(J_V,V)$ for all $n \geq 0$.
\end{lemma}

Let us show how the theorem follows from the lemma. 
Suppose the lemma holds. Replacing $K$ by a finite extension, 
we may assume that $H_V / J_V$ acts unipotently on 
$H^n(J_V,V)$ for all $n \geq 0$. 
We prove the vanishing by induction 
on $n$. The case $n = 0$ is already known
(cf.\ \cite{Ozeki}, Thm.\ 2.1-(1)). 
By Theorem \ref{thm0}, we know
that $V$ has vanishing $H_V$-cohomology. 
Now let $n \geq 1$ and assume that 
$H^{m}(J_V,V) = 0$ for all $1 \leq m < n$.
Then the Hochschild-Serre spectral sequence
(cf.\ \cite{HS}, Thm.\ 2) gives the following
exact sequence:
\[ 0 \rightarrow H^n(H_V / J_V, V^{J_V}) 
\rightarrow H^n(H_V, V) \rightarrow 
H^0(H_V / J_V, H^n(J_V,V)) \rightarrow 
H^{n+1}(H_V / J_V, V^{J_V}). \]
As the second and last terms both vanish, 
we have
$H^0(H_V / J_V, H^n(J_V,V))=0$. The 
vanishing of $J_V$-cohomology
follows from Lemma \ref{lemC} since 
$H_V / J_V$ acts unipotently on $H^n(J_V,V)$.

It remains to prove Lemma \ref{lemA}. 

\begin{proof}[(Proof of Lemma \ref{lemA})]
Let $K^{\text{ur}}$ be the maximal unramified
extension of $K$ and put 
$N_{\infty} = K(E_\infty) \cap K^{\text{ur}}(\mu_{p^\infty})$.
We may view $N_\infty$ as the extension of $K$ obtained by 
adjoining all $p$-power roots of unity to the the 
maximal subextension $N$ of $K(E_\infty)$ which is unramified over $K$. 
Let $M' = M \cap N_\infty$ and we put $G := \Gal(M / K)$,
$H := \Gal(M / M')$ and $Y := \Gal(M' / K) = G/H$. 
Note that $M$ is totally ramified over $M'$. 
In fact, $M'$ is the extension of $K$ obtained
by adjoining all $p$-power roots of unity in $K(E_\infty)$ 
to the maximal subextension of $M$ that is unramified over $K$. 
As $M$ is a $p$-adic Lie extension over $K$
and its residue field $k_M$ is 
potentially prime-to-$p$ over $k$, we see that $k_M$
is a finite extension of $k$ 
(cf.\ Remark \ref{prime-to-p-remark} (ii))
and that $M'$ is of 
finite degree over $K(\mu_{p^\infty})$. 
We have the following diagram of fields:
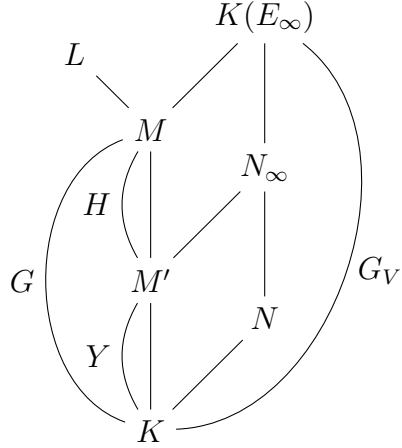
\begin{center}
\begin{tikzpicture}[%
  back line/.style={densely dotted},
  cross line/.style={preaction={draw=white, -,line width=6pt}}]
  \node (A1) {$K(E_\infty)$};
  \node [below of=A1] (B1) {$N_\infty$}; 
  \node [below of=B1] (B2)
{$N$};
  \node (A2) [left of=A, below of=A, node distance=1.5cm] {$M$};
  \node [below of=A2] (B3) {$M'$};
  \node [below of=B3] (C2) {$K$};
  \node (C) [left of=A2, above of=A2, node distance=1cm] {$L$};
  
  \draw[cross line] (C) -- (A2) -- (B3) -- (C2) -- (B2) -- (B1);
  \draw[cross line] (A2) -- (A1) -- (B1) -- (B3);
  \draw[-, bend left = 70pt] (A1) to node [right]{$G_V$} (C2);
  \draw[-, bend right = 30pt] (A2) to node [left]{$H$} (B3);
  \draw[-, bend right = 30pt] (B3) to node [left]{$Y$} (C2);
  \draw[-, bend right = 70pt] (A2) to node [left]{$G$} (C2);
  \end{tikzpicture}
\end{center}
We need an explicit description of the 
action of $Y$ on $H$.
The diagram of fields shown above clearly 
induces the following commutative diagram, 
having exact rows and surjective vertical 
maps:
\begin{equation} \label{commdiagram} 
\begin{tikzpicture} [baseline=(current  bounding  box.center)]
\matrix(m)[matrix of math nodes,
row sep=3em, column sep=2.5em,
text height=1.5ex, text depth=0.25ex]
{1 & \Gal(K(E_\infty)/N_\infty) & G_V & \Gal(N_\infty/K) & 1 \\
1 & H & G & Y & 1\\};
\path[->,font=\scriptsize]
(m-1-1) edge (m-1-2)
(m-1-2) edge (m-1-3)
(m-1-3) edge (m-1-4)
(m-1-4) edge (m-1-5)
(m-2-1) edge (m-2-2)
(m-2-2) edge (m-2-3)
(m-2-3) edge (m-2-4)
(m-2-4) edge (m-2-5);
\path[->>,font=\scriptsize]
(m-1-2) edge (m-2-2)
(m-1-3) edge (m-2-3)
(m-1-4) edge (m-2-4);
\end{tikzpicture}
\end{equation} 
Moreover the above diagram is compatible with 
the actions by inner automorphism, in the 
sense that $\sigma \cdot h = g h g^{-1}$ for 
$\sigma \in Y$, $h \in H$ and a lifting $g$ of 
$\sigma$ to $G$, if and only if
$\tilde{\sigma} \cdot \tilde{h} = \tilde{g} \tilde{h} \tilde{g}^{-1}$, 
for liftings $\tilde{\sigma}$ 
(resp.\ $\tilde{g}$, $\tilde{h}$) of $\sigma$ 
(resp.\ $g$, $h$) to $\Gal(N_\infty/K)$ 
(resp.\ $G_V$, $G_V$). 
In order to obtain the desired explicit 
description for the action of $Y$ on $H$, 
we will use the following well-known result 
(see for instance Proposition 2.6 in 
\cite{Ozeki} which is formulated in a more 
general fashion):
\begin{prop}\label{PropA}
Let $E$ be an elliptic curve over $K$ with good 
ordinary reduction. For some suitable basis of $T_p(E)$, the 
representation $\rho_E$ has the form 
\[ \begin{pmatrix}
\varphi & a  \\
0 & \psi 
\end{pmatrix}, \]
where 
\begin{enumerate}
	\item[(i)] $\varphi : G_K \rightarrow \mathbb{Z}_p^{\times}$ 
	is a continuous character, 
	\item[(ii)] $\psi : G_K \rightarrow \mathbb{Z}_p^{\times}$ 
	is an unramified continuous character and
	\item[(iii)] $a : G_K \rightarrow \Z_p$ is a continuous map.
\end{enumerate}
Moreover, $\chi = \varphi \cdot \psi$.
In particular, the restriction $\varphi |_{I_K}$ 
of $\varphi$ to the inertia subgroup
$I_K$ of $G_K$ coincides with the $p$-adic cyclotomic
character.
\end{prop} 
\noindent 
In fact, the character $\varphi$ in the proposition is
the homomorphism which gives the action of $G_V$ on the Tate module 
$T_p(\mathcal{E}(p)^0)$, while $\psi$ is the homomorphism
giving the action of $G_V$ on $T_p(\mathcal{E}(p)^{\text{\'et}})$.
Here, $\mathcal{E}(p)$ is the $p$-divisible group associated
with the N{\'er}on model of $E$ over $\mathcal{O}_K$ and 
the superscripts ``$0$" and ``{\'e}t" indicate 
the connected $p$-divisible subgroup and 
{\'e}tale quotient of $\mathcal{E}(p)$, respectively.
We now choose a basis of the Tate module $T_p(E)$
of $E$ with respect to which the action
of $G_V$ on $T_p(E)$ is given as in the 
previous proposition.
Let $h \in H$ and $\sigma \in Y$. 
Let $g$ be a lifting of $\sigma$ to $G$.
Let $\tilde{g}$ (resp.\ $\tilde{h}$) be a lifting of $g$ 
(resp.\ $h$) to $G_V$. Note that 
$\tilde{h} \in \Gal(K(E_\infty)/N_\infty)$.
In particular, we have
$\varphi(\tilde{h}) = 1$ and $\psi(\tilde{h}) = 1$. 
A matrix calculation gives
\[ \rho_E(\tilde{g} \cdot \tilde{h})
= \rho_E(\tilde{g}) \rho_E(\tilde{h}) \rho_E(\tilde{g})^{-1} 
= \begin{pmatrix}
1 & \varphi(\tilde{g}) a(\tilde{h}) \psi(\tilde{g})^{-1} \\
0 & 1 
\end{pmatrix}. \]
Since $\varphi = \chi \cdot \psi^{-1}$,
the above equation becomes
\[ \rho_E(\tilde{g}) \rho_E(\tilde{h}) \rho_E(\tilde{g})^{-1} 
= \begin{pmatrix}
1 & a(\tilde{h}) \\
0 & 1 
\end{pmatrix}^{\chi\cdot\psi^{-2}(\tilde{g})}. \]
By the compatibility of the diagram (\ref{commdiagram})
with the action by inner automorphisms,
the preceding equation gives the relation 
\begin{equation*}
\sigma \cdot h = h^{\varepsilon(\sigma)} 
\end{equation*}
where $\varepsilon = \chi\cdot\psi^{-2}$.
Since $\psi$ is unramified and the image of 
$\chi$ is open in $\Z_p^{\times}$, we deduce 
that the image of the character $\varepsilon$ 
is an open subgroup of $\Z_p^{\times}$.
The desired result follows from Lemma \ref{lemB}. 
This completes the proof of Lemma \ref{lemA}
and of Theorem \ref{thm2A}.
\end{proof}

\begin{remark}
\normalfont
The $p$-adic Galois representation $V$ of Theorem \ref{thm2A}
is ordinary in the sense of Definition \ref{ordrepn}
as we can see from Proposition \ref{PropA}. 
We note that the proof of Theorem \ref{thm2A}
can be replicated for ordinary $p$-adic Galois 
representations whose filtration has the same shape as that for
the Tate module of elliptic curves with good 
ordinary reduction. For our purposes, we content ourselves with
the theorem above. 
\end{remark}

On the other hand, we may obtain 
some vanishing results with respect to a more particular 
family of $L$ for more general higher-dimensional ordinary 
representations coming from geometry as follows:

\begin{thm}\label{thm3} 
Let $X$ be a proper smooth variety over $K$ with potential
good ordinary reduction and let $E/K$ be an elliptic curve
with potential good supersingular reduction. 
Let $i$ a positive odd integer and we put
$V = H^{i}_{\text{\'et}} (X_{\overline{K}}, \mathbb{Q}_p)$
and $V' = V_p(E)$.
Then $V$ and $V'$ are cohomologically coprime.
\end{thm}

\begin{proof}
To prove the theorem we have to show that the following 
statements hold: \\
(a) If we put $L = K(E_{\infty})$ and $J_V = \rho(G_L)$, 
	then $V$ has vanishing $J_V$-cohomology; and \\	
(b) If we put $L' = K(V)$ and $J_{V'} = \rho_E(G_{L'})$,	
	then $V'$ has vanishing  $J_{V'}$-cohomology. \\
We only prove statement (a) since statement (b) can be proved
in a similar manner.	
Replacing $K$ with a finite extension, we may assume that 
$X$ has good ordinary reduction and $E$ has good supersingular 
reduction over $K$. We may also assume that $K(V)$ contains
$K(\mu_{p^\infty})$ by extending $K$ further 
(cf. Lemma \ref{lemma1}). 
Let $K^{\text{ur}}$ be the maximal unramified extension 
of $K$ in $\overline{K}$ and put 
$N_{\infty} = K(V) \cap K^{\text{ur}}(\mu_{p^\infty})$. 
The assumption on $V$ implies that the inertia subgroup 
$I_K$ of $G_K$ acts on the associated graded quotients 
$\text{gr$^r$ } V$ by the $r$th power of the $p$-adic
cyclotomic character. In particular the group 
\[ \Gal(K(V)/N_{\infty}) \simeq 
\Gal(K^{\text{ur}}(V)/K^{\text{ur}}(\mu_{p^\infty})) \] 
acts unipotently on $V$. Hence, $\Lie(\Gal(K(V)/N_{\infty}))$
is a nilpotent Lie algebra contained in
$\Lie(H_V) = \Lie(\Gal(K(V)/K_{\infty,V}))$. 
Recall that we may identify $\Gal(L/K)$ with the subgroup 
$\rho_E(G_K)$ of $\Aut(T_p(E)) \simeq \GL_2(\mathbb{Z}_p)$. 
Put $\mathfrak{g} = \Lie(\Gal(L/K))$ 
and $\mathfrak{h}= \Lie(\Gal(L/K(\mu_{p^\infty})))$. 
If $E$ has no formal complex multiplication, then 
$\mathfrak{g} \simeq \mathfrak{gl}_2(\mathbb{Q}_p)$ 
by Proposition \ref{prop3} $(ii)$ and so 
$\mathfrak{h} \simeq \mathfrak{sl}_2(\mathbb{Q}_p)$.
In particular, $\mathfrak{h}$ is simple. 
It immediately follows from Theorem \ref{main} 
that $V$ has vanishing $J_V$-cohomology.
This proves (a) when $E$ has no formal complex multiplication.
We now suppose that $E$ has formal complex multiplication. 
We claim that $M = K(V) \cap L$ is a finite extension 
of $K(\mu_{p^\infty})$.
The restriction map induces a surjection
\[ \Gal(K(V)/N_\infty) \twoheadrightarrow 
\Gal(MN_\infty / N_\infty) \simeq \Gal(M / {M \cap N_\infty}), \]
from which we obtain a surjection of Lie algebras
\[ \Lie(\Gal(K(V)/N_\infty)) \twoheadrightarrow 
\Lie(\Gal(M / {M \cap N_\infty})). \]
As $\Lie(\Gal(K(V)/N_\infty))$ is nilpotent,
we see that $\Lie(\Gal(M / {M \cap N_\infty}))$ is a nilpotent 
subalgebra of $\Lie(\Gal(M / K(\mu_{p^\infty})))$.  
Since $E$ has formal complex multiplication,
we know from Proposition \ref{prop3} $(i)$ that 
$\mathfrak{g}$ is a non-split Cartan subalgebra of 
$\End(V_p(E)) \simeq \mathfrak{gl}_2(\mathbb{Q}_p)$. 
Thus $\mathfrak{g}$ contains the center $\mathfrak{c}$ 
of $\mathfrak{gl}_2(\mathbb{Q}_p)$ and 
$\mathfrak{h} \simeq \mathfrak{g}/\mathfrak{c}$ is a
Cartan sub-algebra of $\mathfrak{sl}_2(\mathbb{Q}_p)$
(cf.\ \cite{Bour2}, Ch.7, \textsection 2, Proposition 5).
Its elements are semisimple in $\mathfrak{sl}_2(\mathbb{Q}_p)$ by 
(\emph{op.\ cit.}, \textsection 4 Thm.\ 2).
Thus, the elements of $\Lie(\Gal(M/K(\mu_{p^\infty})))$ 
are also semisimple  
since it is a quotient of $\mathfrak{h}$. 
Since the Lie algebra $\Lie(\Gal(M/M \cap N_\infty))$ 
is a nilpotent factor of $\Lie(\Gal(M/K(\mu_{p^\infty})))$,
we then see that $\Lie(\Gal(M / M \cap N_\infty)) = 0$. 
This means $M / M \cap N_\infty$ is a finite extension. 
But note that $M \cap N_\infty$ is unramified over $K(\mu_{p^\infty})$. 
Since $\rho_E(I_K)$ is open in $\rho_E(G_K)$ 
again by Proposition \ref{prop3} $(i)$, 
$M \cap N_\infty$ is finite over $K(\mu_{p^\infty})$.
Thus $M / K(\mu_{p^\infty})$ is a finite extension. 
As we remarked earlier, the finiteness of $[M:K(\mu_{p^\infty})]$
is equivalent to the finiteness of the index of $J_V = \rho(G_L)$
in $H_V$. By Lemma \ref{lemma6}, we conclude 
that $V$ has vanishing $J_V$-cohomology.   
\end{proof}

\begin{remark}\label{rem3}          
\normalfont
In Theorem \ref{thm3} when the elliptic curve $E$  
has potential good ordinary reduction, the vanishing 
statement (b) may not hold because $H^0(J_{V'}, {V'})$ 
may be nontrivial. 
This is easily observed by taking $X=E$ and considering
$V=H^1_{\text{\'et}}(X_{\overline{K}}, \Q_p)$.
This observation in fact holds in a more general case. 
Indeed, take any abelian variety $A/K$ with 
potential good ordinary reduction and consider 
$V = V_p(A) \simeq H^{1}_{\text{\'et}} (A_{\overline{K}}, \mathbb{Q}_p)^{\vee}$. 
Since $A$ has potential good ordinary reduction, 
the field $L' = K(A_\infty)$ contains an unramified 
$\mathbb{Z}_p$-extension. 
Hence, the residue field $k_{L'}$ is not a 
potential prime-to-$p$ extension over $k$. 
Replacing $K$ and $L'$ with appropriate finite extensions 
(so that the hypothesis of Corollary \ref{van-ord-abelvar}
is satisfied), we conclude 
that the group $V'$ does not have
vanishing $J_{V'}$-cohomology.
Thus $V$ and $V'$ are not cohomologically coprime.
\end{remark}

\section{Vanishing result for elliptic curves} \label{sec:vanish for EC}

In this section, we determine the cohomological
coprimality of two Galois representations $V_p(E)$ 
and $V_p(E')$ given by elliptic curves $E$ and $E'$,
respectively.

\subsection{The case of good reduction} \label{subsec:good reduction}

We first treat the case where $E$ and $E'$ both have 
potential good reduction over $K$. 
The main result in this subsection is the following:

\begin{thm}\label{thm6} 
Let $E$ and $E'$ be elliptic curves 
with potential good reduction over $K$. 
Put $L=K(E'_\infty)$. 
Then the representations $V_p(E)$ and $V_p(E')$
are cohomologically coprime if one of the 
following conditions is satisfied:
\begin{itemize}
		\item[(i)] $E$ has potential good ordinary reduction and 
		$E'$ has potential good supersingular reduction,
		or vice versa;	
		\item[(ii)] $E$ has potential good supersingular reduction 
		with formal complex multiplication and $E'$ has potential 	
		good supersingular reduction without formal complex multiplication, 
		or vice versa;
		\item[(iii)] $E$ and $E'$ both have potential good 
		supersingular reduction with formal complex multiplication 
		and the group $E(L)[p^\infty]$ of $p$-power division 
		points of $E$ over $L$ is finite;
		\item[(iv)] $E$ and $E'$ both have potential good 
		supersingular reduction without formal complex multiplication 	
		and the group $E(L)[p^\infty]$ is finite.
	\end{itemize}
\end{thm}

By symmetry, it suffices to verify the
following: 

\begin{thm}\label{thm4} 
Let $E$ and $E'$ be elliptic curves 
with potential good reduction over $K$. 
Put $L = K(E'_\infty)$. 
If one of the conditions (i) - (iv)
in Theorem \ref{thm6} is satisfied 
then $V=V_p(E)$ has vanishing $J_V$-cohomology, 
where $J_V = \rho_E (G_L)$.
\end{thm}

\begin{remark}\label{rem4}
\normalfont
Ozeki gave necessary and sufficient conditions 
for the finiteness of $E(L)[p^\infty]$ in \emph{(iii)} and \emph{(iv)}. 
See Propositions 3.7 and 3.8 of \cite{Ozeki} for more details. 
In general, if $E$ is an elliptic curve with good supersingular 
reduction over $K$ and $L$ is a Galois extension of $K$, 
the group $E(L)[p^\infty]$ is finite if and only if $K(E_{\infty})$ 
is not contained in $L$ (\emph{op.\ cit.}, Lemma 3.1).
\end{remark}

The case \emph{(i)} of the theorem is already  
covered by Theorem \ref{thm3}. For case \emph{(ii)}, we may replace $K$ 
with a finite extension so that $E$ and $E'$ both 
have good supersingular reduction over $K$. 
Put $\mathfrak{h} = \Lie(\Gal(K(E_\infty)/K(\mu_{p^\infty})))$.
The Lie algebra of $\Gal(L/K)$ 
is isomorphic to 
$\End(V_p(E')) \simeq \mathfrak{gl}_2(\mathbb{Q}_p)$ 
by Proposition \ref{prop3} \emph{(ii)}.
The Lie algebra 
$\mathfrak{h}' = \Lie(\Gal(L/K(\mu_{p^\infty})))$ 
is isomorphic to $\mathfrak{sl}_2({\mathbb{Q}_p})$. 
In particular, $\mathfrak{h}'$ is simple. 
As $\mathfrak{h}$ is abelian, 
we see that $\mathfrak{h}$ and 
$\mathfrak{h}'$ have no common simple factor.
By Theorem \ref{main}, the desired result follows. 

In view of Corollary \ref{cor1}, to prove
the case of $(iii)$ and $(iv)$,
it suffices to show that the field $K(V_p(E)) \cap L$ is a finite 
extension of $K(\mu_{p^\infty})$. 
We obtain this by the following lemma.

\begin{lemma}\label{lemma2} 
Let $E$ and $E'$ be elliptic curves over $K$ 
which have potential good supersingular reduction.
Suppose $E$ and $E'$ both have 
formal complex multiplication 
or both does not have 
formal complex multiplication. 
Assume further that 
$E(L)[p^\infty]$ is a finite group. 
Then $M := K(E_\infty) \cap L$ 
is a finite extension of $K(\mu_{p^\infty})$.
\end{lemma}

\begin{proof}
We split the proof into two cases: \\
(Case 1) Assume that both $E$ and $E'$ have 
formal complex multiplication.
The Lie algebra $\Lie(\rho_E(G_K))$ 
attached to $E$ is $2$-dimensional, by
Proposition \ref{prop3} \emph{(i)}. 
Thus $\Gal(K(E_\infty)/K)$ is 
a $2$-dimensional $p$-adic Lie group and so 
$\Gal(K(E_\infty)/K(\mu_{p^\infty}))$ is 
$1$-dimensional. 
The same statements hold when $E$ is replaced
by $E'$. 
Replacing $K$ with a finite extension, we may
assume that $\Gal(K(E_\infty)/K(\mu_{p^\infty}))$
is isomorphic to $\mathbb{Z}_p$. If 
$M$ is infinite over $K(\mu_{p^\infty})$, 
then $\Gal(K(E_\infty)/M)$ is of infinite index 
in $\Gal(K(E_\infty)/K(\mu_{p^\infty}))$.
Since the only closed subgroup of $\mathbb{Z}_p$ of 
infinite index is the trivial subgroup, the group 
$\Gal(K(E_\infty)/M)$ must be trivial, and thus $K(E_\infty) = M$.  
That is, $K(E_\infty)$ is contained in $L$. 
Hence, $E(L)[p^\infty]$ is infinite (see Remark \ref{rem4}). 
This contradicts our hypothesis. Therefore, $M$ is a finite 
extension of $K(\mu_{p^\infty})$. \\
(Case 2) Suppose both $E$ and $E'$ does not
have formal complex multiplication.
Put $\mathfrak{g} = \Lie(\Gal(K(E_\infty)/K))$ and 
$\mathfrak{h} = \Lie(\Gal(K(E_\infty)/K(\mu_{p^\infty})))$. 
Then $\mathfrak{g}$ (resp.\ $\mathfrak{h}$) is 
isomorphic to $\mathfrak{gl}_2 (\mathbb{Q}_p)$ 
(resp.\ $\mathfrak{sl}_2 (\mathbb{Q}_p)$). 
In particular, $\mathfrak{h}$ is simple. 
The Lie algebra $\mathfrak{j} = \Lie(\Gal(K(E_\infty)/M))$ 
is an ideal of $\mathfrak{h}$ since 
$\Gal(K(E_\infty)/M)$ is a normal subgroup of 
$\Gal(K(E_\infty)/K(\mu_{p^\infty}))$.
Thus $\mathfrak{j}$ is either $(0)$ or 
$\mathfrak{sl}_2 (\mathbb{Q}_p)$. 
In the former case, $\Gal(K(E_\infty)/M)$
is a finite group and thus $K(E_\infty)/M$ is a finite extension. 
Replacing $K$ with a finite extension, we have 
$K(E_\infty)= M \subset L$. But then 
this implies $E(L)[p^\infty]$ is infinite,
in contrast to our hypothesis. Thus 
$\mathfrak{j} = \mathfrak{sl}_2 (\mathbb{Q}_p)$, 
which means that $\Gal(K(E_\infty)/M)$ is 
an open subgroup of $\Gal(K(E_\infty)/K(\mu_{p^\infty}))$. 
This completes the proof of the lemma and of 
Theorem \ref{thm4}.
\end{proof}

Theorem \ref{thm4} gives another proof of some 
finiteness results in $\cite{Ozeki}$. For instance, in view 
of Remark \ref{rem2}, condition $(i)$ of Theorem \ref{thm4} 
implies a part of Proposition 3.2 in \cite{Ozeki}. 
We also obtain the following corollary. 

\begin{cor}\label{cor_ss} 
Let $E$ and $E'$ be elliptic curves with potential 
good supersingular reduction over $K$, $L = K(E'_\infty)$ and $L' = K(E_\infty)$. 
Put $V=V_p(E)$, $V'=V_p(E')$, $J_V = \rho_E(G_L)$ and $J_{V'} = \rho_{E'}(G_{L'})$. Assume that $E$ and $E'$ 
both have formal complex multiplication or 
both do not have formal complex multiplication. 
Then the following statements are equivalent:
\\
(1) 
$V$ and $V'$ are cohomologically 
coprime 
\\
(2) 
$L \cap L'$ is a finite extension
of $K(\mu_{p^\infty})$,
\\
(3) 
$V$ has vanishing $J_V$-cohomology,
\\
(3') 
$V'$ has vanishing $J_{V'}$-cohomology, 
\\
(4)
$E(L'')[p^\infty]$ is a finite group 
for any finite extension $L''$ of $L$,
\\
(4')
$E'(L'')[p^\infty]$ is a finite group
for any finite extension $L''$ of $L'$, and
\\
(5)
The $p$-divisible groups $\mathcal{E}(p)$ 
and $\mathcal{E}'(p)$
attached to $E$ and $E'$, respectively, 
are not isogenous over $\mathcal{O}_{K'}$ 
for any finite extension $K'$ of $K$. 
\end{cor}

\begin{proof}
It remains to prove the equivalence of 
each of the first six conditions with the last one.
Replacing $K$ by a finite extension, we may
assume that $E$ and $E'$ have good supersingular
reduction over $K$.
We prove the equivalence $(4) \Leftrightarrow (5)$.
If $E$ and $E'$ both do not have 
formal complex multiplication 
then this equivalence 
is given by Proposition 3.8 in \cite{Ozeki}.
Assume that $E$ and $E'$ both have 
formal complex multiplication. 
Let $L''$ be a finite extension of $L$
such that $E(L'')[p^\infty]$ is infinite.
Replacing $K$ by a finite extension,
we may assume that $L=L''$.
Then Proposition 3.7 in \cite{Ozeki} 
implies that $E$ and $E'$ have
the same fields of formal complex multiplication,
say $F$. 
The representations
$\rho_E : G_K \rightarrow \GL(V_p(E))$ and 
$\rho_{E'} : G_K \rightarrow \GL(V_p(E'))$
factor through $\Gal(K^{\mathrm{ab}}/K)$,
where $K^{\mathrm{ab}}$ denotes the 
maximal abelian extension of $K$. 
Moreover $\rho_E$ and $\rho_{E'}$ both have values 
in $\mathcal{O}^{\times}_{F}$ 
and their restrictions to the inertia group
are respectively given by
\[ \rho_E|_{I_K}, \rho_{E'}|_{I_K} 
: I(K^{\mathrm{ab}}/K)  \simeq \mathcal{O}^{\times}_{K} 
\rightarrow \mathcal{O}^{\times}_{F}. \]
Here, $K^{\mathrm{ab}}$ denotes the 
maximal abelian extension of $K$ and 
$I(K^{\mathrm{ab}}/K)$ is the
inertia subgroup of
$\Gal(K^{\mathrm{ab}}/K)$, with the isomorphism 
$I(K^{\mathrm{ab}}/K) \simeq \mathcal{O}_{K^\times}$
coming from local class field theory.
In fact, $\rho_E|_{I_K}$ and $\rho_{E'}|_{I_K} $
are equal since they are both given
by the map 
$x \mapsto \mathrm{Nr}_{K/F}(x^{-1})$, 
where 
$\mathrm{Nr}_{K/F} : K^{\times} \rightarrow F^{\times}$
is the norm map (cf.\ \cite{Ser1}, Chap.\ IV, A.2.2). 
From Proposition \ref{prop3} (\emph{i}), 
we know that $\rho_E(I_K)$ (resp. $\rho_{E'}(I_K)$) 
is an open subgroup of $\rho_E(G_K)$ 
(resp. $\rho_{E'}(G_K)$).
This, together with the assumption that 
$E(L)[p^\infty]$ is infinite implies that
$\rho_E(I_{K'}) = \rho_E(G_{K'}) = \rho_{E'}(G_{K'}) = \rho_{E'}(I_{K'})$
after a finite extension $K'/K$.
We see that the Tate modules $T_p(E)$ and $T_p(E')$
become isomorphic over $K'$.
By a well-known result due to Tate 
(cf.\ \cite{Tate}, Corollary 1),
the $p$-divisible groups $\mathcal{E}(p)$ 
and $\mathcal{E}'(p)$ are isogenous over $\mathcal{O}_{K'}$. Conversely, if there
exists a finite extension $K'$ over $K$ such that $\mathcal{E}(p)$ 
and $\mathcal{E}'(p)$ are isogenous over 
$\mathcal{O}_{K'}$, then $T_p(E)$ and 
$T_p(E')$ are isomorphic over $K'$ which 
shows that $K'(E_\infty) = K'(E'_\infty)$ which 
is a finite extension of $L$.
Therefore we obtain a 
finite extension $L''$ of $L$
such that $E(L'')[p^\infty]$ is infinite.
\end{proof}

\subsection{The case of multiplicative reduction}\label{subsec:multip reduction}

We now treat the case where $E'$ has potential multiplicative reduction 
over $K$. For this case, the Lie algebra of $\Gal(K(E'_\infty)/K)$ 
is given by Proposition \ref{prop3} \emph{(v)}. 
We have the following result.

\begin{thm}\label{thm5}		
Let $E$ and $E'$ be elliptic curves over $K$ such that 
$E$ has potential good reduction over $K$ and
$E'$ has potential multiplicative reduction over $K$. 
Put $L = K(E'_\infty)$. Then $V = V_p(E)$ has vanishing 
$J_V$-cohomology, where $J_V = \rho_E (G_L)$.
\end{thm}

\begin{proof}
Replace $K$ with a finite extension so that 
$E$ and $E'$ have good and multiplicative 
reductions over $K$, respectively.
We first note that the residue field $k_L$ of $L$ 
is a potential prime-to-$p$ extension of $k$ since 
$\Lie(\rho_{E'}(G_K)) = \Lie (\rho_{E'}(I_K))$.
Thus, the case where $E$ has good ordinary reduction 
is just a consequence of Theorem \ref{thm2A}.
It remains to settle the case where $E$ has
good supersingular reduction over $K$.
If $E$ has no formal complex multiplication, 
note that the Lie algebra 
$\mathfrak{h}_1 = \Lie(\Gal(K(E_\infty) / K(\mu_{p^\infty}))) \simeq \mathfrak{sl}_2(\mathbb{Q}_p)$ 
is simple. On the other hand, the Lie algebra
$\mathfrak{h}_2 = \Lie(\Gal(K(E'_\infty) / K(\mu_{p^\infty})))$
is a abelian. The result follows from Theorem \ref{main}.
If $E$ has formal complex multiplication, 
then by virtue of Corollary \ref{cor1}, 
it suffices to prove that $M := L \cap K(E_\infty)$ is a 
finite extension of $K(\mu_{p^\infty})$. 
The Lie algebra $\Lie(\rho_E(G_K))$ 
attached to $E$ is $2$-dimensional, by
Proposition \ref{prop3} \emph{(i)}. 
Thus $\Gal(K(E_\infty)/K)$ is 
a $2$-dimensional $p$-adic Lie group and so 
$\Gal(K(E_\infty)/K(\mu_{p^\infty}))$ is 
$1$-dimensional. 
As in the proof for Case (1) of 
Lemma \ref{lemma2}, if we assume that $M$ is of 
infinite degree over $K(\mu_{p^\infty})$, then 
the group $\Gal(K(E_\infty)/M)$ must be trivial, 
and thus $K(E_\infty) = M$.  
That is, $K(E_\infty)$ is contained in $L$.
Thus we have a natural surjection 
$\Gal(L/K) \twoheadrightarrow \Gal(K(E_\infty)/K)$
which induces a surjection of Lie algebras
$\Lie(\Gal(L/K)) \twoheadrightarrow \Lie(\Gal(K(E_\infty)/K))$.
Since both Lie algebras are two-dimensional, 
the above surjection of Lie algebras must be an isomorphism.
In view of Proposition \ref{prop3}
\emph{(i)} and \emph{(v)}, we have a contradiction. 
Therefore, $M$ is a finite extension of 
$K(\mu_{p^\infty})$.
\end{proof}

\begin{remark}\label{cohom-coprime-mult}
\normalfont
Despite the above result, we cannot 
expect much about the cohomological coprimality
of $V=V_p(E)$ and $V'=V_p(E')$ if at least
one of $E$ and $E'$ has multiplicative reduction 
over $K$. For instance if $E'$ has split
multiplicative reduction, the theory
of Tate curves shows that 
$H^0(H_{V'}, V')$ is non-trivial.
On the other hand if $E'$ has non-split
multiplicative reduction, we are not
certain if all the $J_{V'}$-cohomology 
groups of $V'$ vanish or not.
(But see Proposition 3.10 in \cite{Ozeki} for 
conditions where $H^0(J_{V'}, V')$
vanishes).
\end{remark}

\section{Cohomologies of Global Representations}

In this section, we discuss some global 
analogues of the results we obtained in
the previous sections. In fact, these are
consequences of the local results that we proved.
Let $F$ be an algebraic number field, 
that is, a finite extension of $\Q$.
We now consider a proper smooth algebraic 
variety $X$ defined over $F$. 
Denote by $X_{\overline{F}}$ the extension
of scalars of $X$ to $\overline{F}$.
In this section we take
\[ V = H^{i}_{\text{\'et}} (X_{\overline{F}}, \Q_p) \] 
and again denote by $\rho : G_F \rightarrow \GL(V)$ the 
continuous homomorphism giving the action of $G_F$
on $V$. As in the previous sections if 
$V = V_p(E) = T_p(E) \otimes_{\Z_p} \Q_p$ is 
given by an elliptic curve $E$ over $F$, we write $\rho_E$
instead of $\rho$. 
For an algebraic extension $L$ of $F$ 
and a place $w$ of $L$,
we denote by $L_w$ the union of the 
completions at $w$ of all finite 
extensions of $F$ contained in $L$ .

\begin{thm}\label{global_var} 
Let $X$ be a proper smooth variety and $E$ be an elliptic curve 
over $F$. Suppose there is at least one place $v$ of $F$ 
above $p$ such that $X$ has potential good ordinary reduction at 
$v$ and $E$ has potential good supersingular reduction at $v$.
Let $i$ be a positive odd integer. Put $V = H^{i}_{\text{\'et}} (X_{\overline{F}}, \Q_p)$
and $V' = V_p(E)$. Then $V$ and $V'$ are cohomologically coprime.
\end{thm}

\begin{thm}\label{global_ellcur} 
Let $E$ and $E'$ be elliptic curves 
over $F$. Put $L=F(E'_\infty)$. 
Then $V_p(E)$ and $V_p(E')$
are cohomologically coprime if there is at least one 
place $v$ of $F$ above $p$ such that one of the 
following conditions is satisfied:
	\begin{itemize}
		\item[(i)] $E$ has potential good ordinary reduction at 
		$v$	and $E'$ has potential good supersingular reduction at $v$,
		or vice versa;	
		\item[(ii)] $E$ has potential good supersingular reduction at 
		$v$ with formal complex multiplication and $E'$ has potential 	
		good supersingular reduction at $v$ without formal complex 
		multiplication, or vice versa;
		\item[(iii)] $E$ and $E'$ both have potential good 
		supersingular reduction at $v$ 
		with formal complex multiplication and the group 
		$E(L_w)[p^\infty]$ of $p$-power division points of $E$ 
		over $L_w$ is finite for a place $w$ of $L$ 
		lying above $v$;
		\item[(iv)] $E$ and $E'$ both have potential good 
		supersingular reduction at $v$ without formal complex multiplication 	
		and the group $E(L_w)[p^\infty]$ is finite for a place 
		$w$ of $M$ lying above $v$.
	\end{itemize}
\end{thm}

We give the proof of Theorem \ref{global_var} below. 
Theorem \ref{global_ellcur} can be verified 
similarly from the proof of Theorem \ref{thm4}.
\begin{proof}[Proof of Theorem \ref{global_var}]
Put  $J_{V} = \rho(G_L)$ and  $J_{V'} = \rho_E(G_{L'})$,
where $L = F(E_{\infty})$ and $L' = F(V)$.
We must show that $V$ has vanishing $J_V$-cohomology
and $V'$ has vanishing $J_{V'}$-cohomology.
To do this, we proceed as in \cite{CSW}, \S4.4, Example 3. 
For a vector space $W$ over $\Q_p$, we write $W_{\overline{\Q_p}}$
for $W  \otimes_{\Q_p} \overline{\Q_p}$.
Let $w$ (resp.\ $\nu$, $u$) be a place of 
$L$ (resp.\ $L'$, $F(\mu_{p^\infty})$) 
lying above $v$.
By replacing $F_v$ by a finite extension we may 
assume that $L'_\nu$ contains  $F(\mu_{p^\infty})_u=F_v(\mu_{p^\infty})$.
In the proof of Theorem \ref{thm3}, we showed that 
$L_w \cap L'_\nu$ is a finite extension over $F_v(\mu_{p^\infty})$.
We let $D_V$ (resp.\ $\mathscr{D}_V$) be the image 
of $\rho$ restricted to the decomposition group 
in $G_L$ (resp.\ $G_{F(\mu_{p^\infty})}$) of some fixed 
place of $\overline{F}$ above $w$ (resp.\ $u$). 
Then $D_V$ is an open subgroup of $\mathscr{D}_V$ and 
so their Lie algebras coincide.
But the Lie algebra $\Lie(\mathscr{D}_V)_{\overline{\Q_p}}$ 
satisfies the strong Serre criterion (cf.\ \cite{CSW}). 
Thus $\Lie(J_V)_{\overline{\Q_p}} \supset \Lie(D_V)_{\overline{\Q_p}}$ 
also satisfies the strong Serre criterion. It
follows from Proposition 2.3 of \emph{op.\ cit.}
that $V$ has vanishing $J_V$-cohomology.
Replacing $(\rho, V, L)$ with $(\rho_E, V', L')$ above,
we may prove in the same manner that $V'$ has vanishing 
$J_{V'}$-cohomology. This completes the proof
of Theorem \ref{global_var}. 
\end{proof}

By arguing in a similar manner as above, 
it can be shown that the proof of Theorem \ref{thm5} 
for the good supersingular reduction case
also implies the following global result. 

\begin{thm}\label{global_ellcur2} 
Let $E$ and $E'$ be elliptic curves 
over $F$. Put $L = F(E'_\infty)$. Then
$V=V_p(E)$ has vanishing $J_V$-cohomology, 
where $J_V = \rho_E(G_L)$, if there is at least one place
$v$ above $p$ such that $E$ has potential good 
supersingular reduction at $v$ and 
$E'$ has potential multiplicative reduction at $v$.  
\end{thm}

Notice that in Theorem \ref{global_ellcur2} 
we did not include the case where $E$ has potential 
good ordinary reduction and $E'$ has potential multiplicative 
reduction. This is because in this case 
$\Lie (D_V)$ may be smaller than $\Lie(\mathscr{D}_V)$, with
the notation above. Nevertheless, the same assertion
for the global representation may be shown to hold 
by considering places $v$ not lying above $p$.
This, along with similar cases, will be treated in a 
subsequent paper.

\noindent	Jerome T. Dimabayao  \\ 
   				Graduate School of Mathematics\\
   				Kyushu University\\
   				744 Motooka, Nishi-ku, Fukuoka\\
   				819-0395, Japan \\
   				\texttt{dimabayao@math.kyushu-u.ac.jp} \\
\noindent	and  \\
\noindent	Institute of Mathematics \\ 
					University of the Philippines \\
					C.P. Garcia St., U.P.Campus Diliman \\
					1101 Quezon City, Philippines \\
					\texttt{jdimabayao@math.upd.edu.ph}
   

\begin{thebibliography}{100}	
\small

\bibitem[BK86]{BK} 
S. Bloch and K. Kato,  
\emph{$p$-adic {\'e}tale cohomology}, 
Inst. Hautes Études Sci. Publ. Math. No. 63 (1986), 107--152.

\bibitem[Bo08]{Bour2} 
N. Bourbaki,  
\emph{Lie Groups and Lie Algebras, Chapters 7--9}, 
Translated from the French, \emph{Elements of Mathematics} (Berlin), 
Springer-Verlag Berlin Heidelberg, 2008.

\bibitem[CLS98]{C-LS} 
Chiarellotto, B., Le Stum B., 
\emph{Sur la puret\'e de la cohomologie cristalline}, 
C.R.  Acad. Sci. Paris \textbf{326}, S\'erie I 1998, pp. 961-963.

\bibitem[CH01]{CH} 
J. Coates and S. Howson,
\emph{Euler Characteristics and elliptic curves II},
Journal of Math. Society of Japan, 
$\textbf{53}$, 2001, 175--235. 

\bibitem[CSS03]{CSS}
J. Coates, P. Schneider, and R. Sujatha,
\emph{Links Between Cyclotomic and $\GL_2$ Iwasawa Theory},
Documenta Mathematica, Extra Volume Kato, 2003, 187--215.

\bibitem[CSW01]{CSW} 
J. Coates, R. Sujatha, and J-P. Wintenberger, 
\emph{On Euler-Poincar$\acute{e}$ characteristics of finite dimensional $p$-adic Galois representations}, 
Inst. Hautes \'Etudes Sci. Publ. Math., 2001, 107--143. 

\bibitem[DDSMS99]{DSMS} 
J.D. Dixon, M.P.F. Du Sautoy, A. Mann, and D. Segal, 
\emph{Analytic Pro-$p$ Groups}, 
Cambridge Studies in Advanced Mathematics $\textbf{61}$, 
Cambridge University Press, 2nd ed., 1999.

\bibitem[Fo94]{Fon}
J.-M. Fontaine, 
\emph{Repr{\'e}sentations $p$-adiques semistables}, 
in \emph{P{\'e}riodes $p$-adiques (Bures-sur-Yvette, France, 1988),} 
(J-M Fontaine ed.), Ast{\'e}risque $\textbf{223}$, 
Soc.Math.France, Montrouge, 1994, 113--184.

\bibitem[HS53]{HS}
G. Hochschild and J.-P. Serre,
\emph{Cohomology of group extensions},
Trans. AMS \textbf{74}, 1953, 110--134.

\bibitem[Il94]{LI}
L. Illusie, 
\emph{R{\'e}duction semi-stable ordinaire, cohomologie {\'e}tale 
$p$-adique et Cohomologie de de Rham, d'apres Bloch-Kato et Hyodo} 
in \emph{P{\'e}riodes $p$-adiques (Bures-sur-Yvette, France, 1988),} 
(J-M Fontaine ed.), Ast{\'e}risque $\textbf{223}$, 
Soc.Math.France, Montrouge, 1994, 209--220.

\bibitem[KT13]{KT} 
Y. Kubo and Y. Taguchi, 
\emph{A generalization of a theorem of Imai and its applications to Iwasawa theory}, 
Mathematische Zeitschrift(2013), DOI 10.1007/s00209-103-1176-3.

\bibitem[Oz09]{Ozeki} 
Y. Ozeki, 
\emph{Torsion Points of Abelian Varieties with Values in Infinite Extensions over a $p$-adic Field}, 
Publ. RIMS, Kyoto University, \textbf{45}, 
2009, 1011--1031.

\bibitem[PR94]{PR}
B. Perrin-Riou, 
\emph{Repr{\'e}sentations $p$-adiques ordinaires} 
in \emph{P{\'e}riodes $p$-adiques (Bures-sur-Yvette, France, 1988),} 
(J-M Fontaine ed.), Ast{\'e}risque $\textbf{223}$, 
Soc.Math.France, Montrouge, 1994, 185--208.

\bibitem[Se71]{Ser3} 
J-P. Serre, 
\emph{Sur les groupes de congruence des vari\'et\'es ab\'eliennes. II}, 
Izv. Akad. Nauk. SSSR Ser. Mat., \textbf{35}, 
1971, 731--735. (={\OE}uvres 2, 686--690)

\bibitem[Se89]{Ser1} 
J-P. Serre, 
\emph{Abelian $l$-adic Representations and Elliptic Curves}, 
With the collaboration of Willem Kuyk and John Labute. Second Edition. 
Advanced Book Classics. Addison-Wesley Publishing Company, 
Advanced Book Program, Redwood City, CA, 1989.

\bibitem[Ta67]{Tate} 
J. Tate,  
\emph{$p$-divisible groups},
Proceedings of a Conference on Local Fields, 
Driebergen, 1966, Springer, 1967, pp. 158--183.

\end{thebibliography}
\end{document}